\documentclass[oneside,12pt]{amsart}
\usepackage[a4paper,margin=1in]{geometry}
\usepackage{graphicx}
\usepackage{amsthm,amssymb,amsmath,amsfonts}
\usepackage{caption}
\usepackage[square,numbers]{natbib}
\usepackage{tikz}
\usepackage{svg}
\usepackage{accents}
\usepackage[toc,page]{appendix}

\usepackage{import}
\usepackage{xifthen}
\usepackage{pdfpages}
\usepackage{transparent}
\usepackage{hyperref}

\numberwithin{equation}{section}

\newtheorem{theorem}{Theorem}[section]
\newtheorem*{theorem*}{Theorem}

\newtheorem{lemma}[theorem]{Lemma}
\newtheorem*{lemma*}{Lemma}

\theoremstyle{definition}

\newtheorem{remark}[theorem]{Remark}

\newcommand{\RN}[1]{%
  \textup{\uppercase\expandafter{\romannumeral#1}}%
}

\begin{document}

\title{Second order estimates for a free boundary phase transition}

\author{Jingeon An}



\address{Department of Mathematics and Computer Science, University of Basel, Spiegelgasse 1, 4052 Basel, Switzerland}

\email{jingeon.an@icloud.com}

\begin{abstract}
    It is well known that minimizers of the Allen-Cahn-type functional
    \[
    J_\epsilon(u):=\int_\Omega\frac{\epsilon|\nabla u|^2}{2}+\frac{W(u)}{\epsilon},
    \]
    where $W$ is a double-well potential, resemble minimal surfaces in the sense that their level sets converge to a minimal surface as $\epsilon\rightarrow 0$. In this work, we consider the indicator potential $W(\tau)=\chi_{(-1,1)}(\tau)$, which leads to the Bernoulli-type free-boundary problem
    \[
        \left\{
            \begin{alignedat}{2}
                \Delta u&=0&\quad&\textrm{in}\quad\{|u|<1\}\\
                |\nabla u|&=\epsilon^{-1}&\quad&\textrm{on}\quad\partial \{|u|<1\}.
            \end{alignedat}
        \right.
    \]
    We provide a short proof that the transition layers are uniformly $C^{2,\alpha}$ regular, up to the free boundary. In addition to the uniform $C^{2,\alpha}$ estimate, we also obtain improved $C^\alpha$ mean curvature bound that decays in an algebraic rate of $\epsilon$, which confirms the convergence of interfaces to the minimal surface in a very strong sense. We present a simple elliptic equation
    \[
        \Delta\phi=H^2-|\mathbf{A}|^2
    \]
    where $\phi=\log(1/|\nabla u|)$ is the log-gradient of $u$, $H$ and $\mathbf{A}$ are the mean curvature and the second fundamental form of level surfaces, respectively. From this, the uniform estimates readily follow. The whole argument is performed in a general Riemannian manifold setting.
\end{abstract}

\maketitle

\section{Introduction}

In this work, we study the free boundary Allen-Cahn equation, derived from the energy functional
\begin{equation}\label{eq:landau-ginzburg}
    J_\epsilon(v):=\int_{\Omega}\frac{\epsilon|\nabla v|^2}{2}+\frac{W(v)}{\epsilon},
\end{equation}
where $\epsilon>0$ is a small parameter, and $W(t)$ is an indicator potential $\chi_{(-1,1)}(t)$ (see Figure \ref{fig:description of potentials}). This functional arises naturally in models of phase transitions---such as those describing the concentration in a binary alloy \cite{allen1972ground,allen1973correction}---where $v$ represents a physical state and $W(v)$ represents its associated potential energy.

\begin{figure}[t]
\includegraphics[width=10cm]{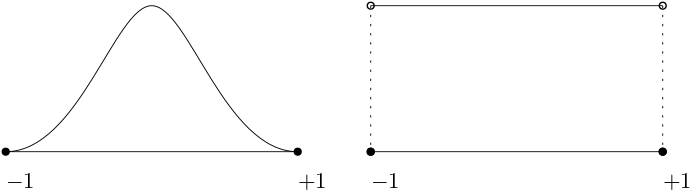}
\centering
\caption{Description of potential $W$. The left side depicts the potential for the traditional Allen-Cahn model, given by $W(t):=(1-t^2)^2/4$. The right illustrates the free boundary version, defined as $W(t)=\chi_{(-1,1)}(t)$.}
\label{fig:description of potentials}
\end{figure}

Since $W(v)$ attains its minimum at $v=\pm 1$, the system energetically favors these states. In the absence of the gradient term $|\nabla u|^2$, any function attaining $\pm 1$ almost everywhere (for example $u=\chi_E-\chi_E^c$ for any measurable set $E$) would minimize the energy functional. However, the gradient term penalizes discontinuities and forces the transition regions to behave like interfaces with a nontrivial surface energy in the minimization of $J_\epsilon$. 

In this sense, the minimization problem of \eqref{eq:landau-ginzburg} is closely related to the theory of minimal surfaces. Consider a minimizer $u_\epsilon$ of $J_\epsilon$ for $\epsilon>0$. Modica's classical result \cite{modica1987gradient} shows that, in the limit $\epsilon\rightarrow 0$, converging subsequences of $(u_\epsilon)_{\epsilon>0}$ in $L^1_\mathrm{loc}$ have their limit 
\[
u=\chi_A-\chi_{A^c},
\]
where $A$ is a set of minimal perimeter in $\Omega$ (i.e., $\partial A\cap\Omega$ is a minimizing minimal surface). 

This connection between phase transition models and minimal surfaces is epitomized by the celebrated Allen-Cahn equation. Namely, the Euler-Lagrange equation of the classical Allen-Cahn energy
\[
J(v)=\int_{\Omega}\frac{|\nabla v|^2}{2}+\frac{(1-v^2)^2}{4},
\]
\begin{equation}\label{eq: Allen-Cahn}
    \Delta u=u^3-u,\,|u|\leq1,\quad\text{ in }\quad\Omega.
\end{equation}

De Giorgi \cite{de1979convergence} proposed in 1978 the following conjecture, known as De Giorgi's conjecture, or $\epsilon$-Bernstein problem: for $n\leq 8$, any global monotone (in one direction) solution of \eqref{eq: Allen-Cahn} has level sets that are hyperplanes. De Giorgi's conjecture is stated only for lower dimensions $n\leq 8$, because for $n\geq 9$, there exists a non-planar minimal surface based on the Simons cone 
\[
\{x_1^2+x_2^2+x_3^2+x_4^2=x_5^2+x_6^2+x_7^2+x_8^2\}\subset\mathbb{R}^8,
\]
due to Del Pino, Kowalczyk, and Wei \cite{del2008counterexample}.

De Giorgi’s conjecture has attracted considerable attention, and numerous partial results have been established. In dimension two and three, the problem is very well understood. In this setting, it has been shown that global monotone solutions are indeed one-dimensional (see \cite{ghoussoub1998conjecture} for $n=2$, and \cite{ambrosio2000entire} for $n=3$). For $n\geq 4$, the deep result of Savin \cite{savin2009regularity} showed that for $4\leq n\leq 8$, the conjecture holds true under the additional limit condition 
\[
    u(x_1,\cdots,x_n)\rightarrow\pm 1,\quad\textrm{as}\quad x_n\rightarrow\pm\infty.
\]
For more comprehensive introduction on Allen-Cahn equation and De Giorgi's conjecture, interested readers may visit the survey of Savin \cite{savin2010phase,savin2010minimal}, and Chan and Wei \cite{chan2018giorgi}.

Similar to the classical Allen-Cahn equation \eqref{eq: Allen-Cahn}, we may consider parametrized potentials $W_\delta(\tau):=(1-\tau^2)^{\delta}$, already introduced by Caffarelli and Cordoba \cite{caffarelli1995uniform}. In this setting, $\delta=2$ becomes the classical Allen-Cahn potential, and cases $\delta\in [0,2)$ were studied, for example, in \cite{du2022four,kamburov2013free,liu2018free,valdinoci2004plane,valdinoci2006flatness,wang2015structure}. By letting $\delta\rightarrow 0$, we obtain the indicator potential $\chi_{(-1,1)}(\tau)$. By perturbing the energy functional 
\[
    J_\epsilon(v):=\int_{\Omega}\frac{\epsilon|\nabla v|^2}{2}+\frac{\chi_{(-1,1)(v)}}{\epsilon},
\]
any critical point $u_\epsilon$ of $J_\epsilon$ satisfies the following free boundary problem in the viscosity sense (see \cite{de2009existence} and cf. \eqref{eq: Allen-Cahn}): 
\begin{equation}\label{epsilon Savin-Kamburov equation}
    \left\{
        \begin{alignedat}{2}
            \Delta u_\epsilon&=0&\quad&\text{in}\quad\{|u_\epsilon|<1\}\\
            |\nabla u_\epsilon|&=1/\epsilon&\quad&\text{on}\quad\partial\{|u_\epsilon|<1\}.
        \end{alignedat}
    \right.
\end{equation}
To provide a geometric illustration of the problem, consider a band with a width comparable to $\epsilon$ (see Figure \ref{fig:Representational image of u_epsilon}). This band is composed of transition layers of $u_\epsilon$, i.e., $\{|u_\epsilon|<1\}$, where $u_\epsilon$ is harmonic. 

\begin{figure}
    \centering
    \includegraphics[width=9cm]{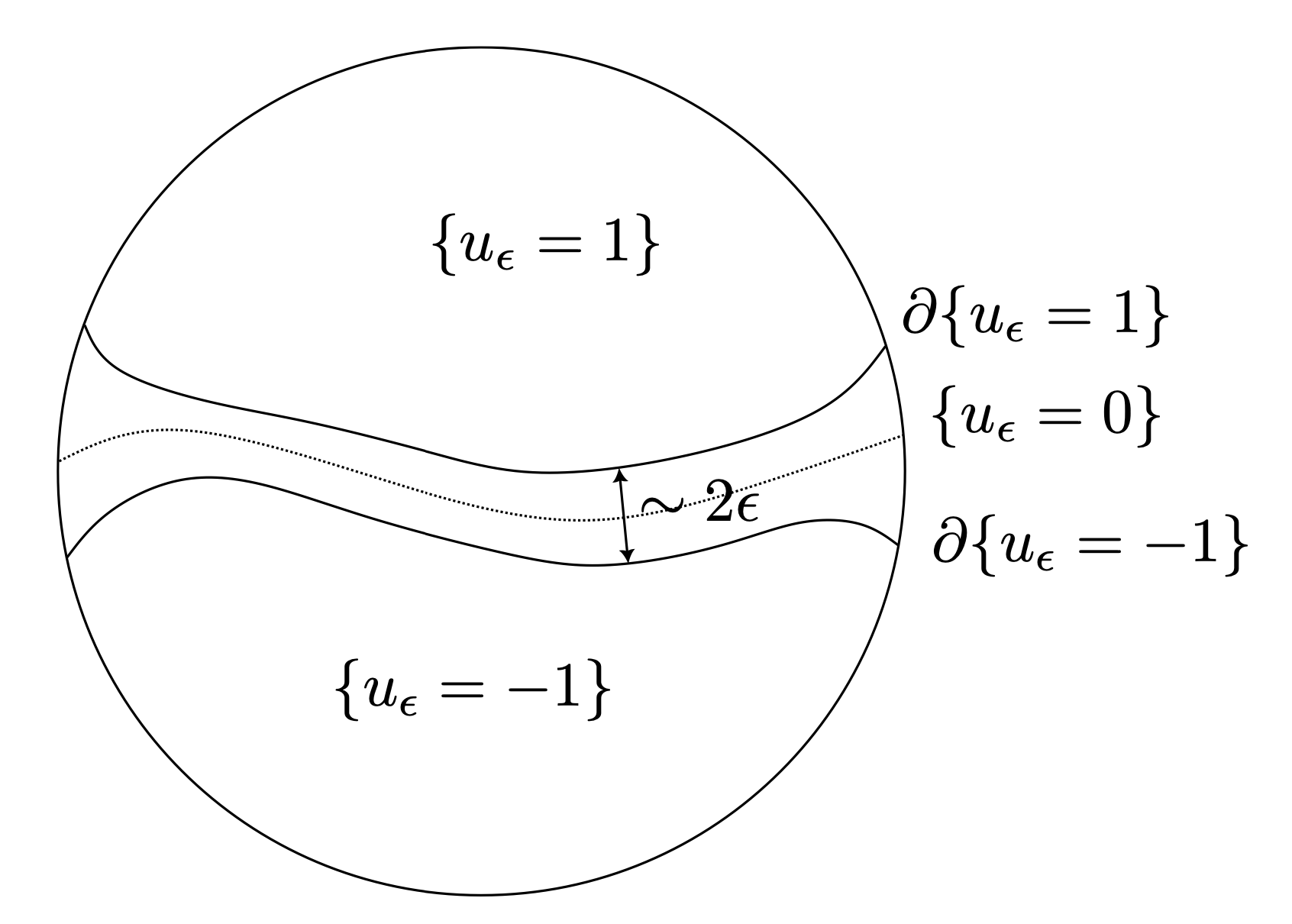}
    \caption{Representational image of $u_\epsilon$}
    \label{fig:Representational image of u_epsilon}
\end{figure}
Mathematically speaking, the free boundary Allen-Cahn, \eqref{epsilon Savin-Kamburov equation}, enjoys a simpler structure than that of classical Allen-Cahn. The ambient function $u_\epsilon$ is harmonic in the transition layers, and all the nonlinearities are ``concentrated" as the free boundary condition $|\nabla u_\epsilon|=1/\epsilon$. It also enjoys the property that separate interfaces (that is to say, the interfaces that are separated by free boundaries) do not interact with each other, unlike the classical Allen-Cahn equation. Therefore, a free boundary Allen-Cahn model can also serve as a natural framework for approximating minimal surfaces, that may be more tractable, in some cases, than the classical Allen-Cahn equation. This motivates the study of properties of free boundary Allen-Cahn equation.

Almost all of the progress made for the classical Allen-Cahn equation has also been made for its free boundary version. In the context of Alt-Caffarelli free boundary problem (introduced by Alt-Caffarelli in 1981 \cite{Alt1981}), defined by
\[
    \left\{
        \begin{alignedat}{2}
            u&\geq 0&\quad&\textrm{in}\quad \{u>0\}\\
            \Delta u&=0&\quad&\textrm{in}\quad \{u>0\}\\
            |\nabla u|&=1&\quad&\textrm{on}\quad\partial\{u>0\},
        \end{alignedat}
    \right.
\]
Caffarelli, Jerison, and Kenig \cite{caffarelli2004global} showed that, in three dimensions, global minimizers of corresponding energy are planar. Jerison and Savin \cite{jerison2015some} further established that minimal cones are hyperplanes in four dimensions, and De Silva and Jerison \cite{de2009singular} provided a counterexample for $n=7$. Kamburov in \cite{kamburov2013free} gave a counter-example to the free boundary analogue of De Giorgi's conjecture for $n=9$. Recently, Basulto and Kamburov \cite{basulto2024one} offered a complete classification of global solutions with finite Morse index in the one-dimensional case, and solutions with graphical free boundaries were studied in \cite{de2011gradient,engelstein2023graphical}.

Most notably, Chan, Fernandez-Real, Figalli, and Serra \cite{chan2025global} have made a significant breakthrough by classifying all stable global solutions to the one-phase free-boundary problem in dimension 3. This achievement has led to the resolution of the free boundary version of De Giorgi’s conjecture in dimension $4$.

Despite all these combined efforts and groundbreaking works, the higher regularity theory of interfaces remains somewhat incomplete. In 2006, Caffarelli and C\'ordoba \cite{caffarelli2006phase} demonstrated, that if the level surfaces are uniformly Lipschitz with respect to $\epsilon$, then the solution possesses a uniform $C^{1,\alpha}$ bound for some $\alpha\in(0,1)$, for both classical and free boundary version of Allen-Cahn equations. Subsequently, Wang-Wei \cite{wang2019second,wang2019finite} resolved the uniform $C^{2,\alpha}$ regularity problem for the classical Allen-Cahn equation, and sheet separation result for $n=2$. This was further extended to general Riemannian manifold setting and $n=3$, by Chodosh-Mantoulidis \cite{chodosh2020minimal}. However, to the best of the author’s knowledge, the extension of this result to the free-boundary Allen-Cahn equation remains an open question.

Precisely, we have the following theorem:
\begin{theorem}\label{thm: main theorem}
    Let $(M,\bar{g})$ be a $n+1$ dimensional smooth Riemannian manifold, $\epsilon\in(0,1]$, $x_0\in (M,\bar{g})$, and $u:B_1^{\bar{g}}(x_0)=:B_1\rightarrow \mathbb{R}$ a function satisfying
    \[
        \left\{
            \begin{alignedat}{2}
                \Delta_{\bar{g}}u&=0\quad&\textup{in}&\quad\{|u|<1\}\subset B_1\\
                |\nabla_{\bar g}u|&=1/\epsilon\quad &\textup{on}&\quad\partial\{|u|<1\}.
            \end{alignedat}
        \right.
    \]
    Assume $x_0\in \{u=0\}$, and that the level sets of $u$, $\{u=\tau\}$ for $\tau\in(-1,1)$ satisfy uniform $C^{1,1}$ regularity for each level surfaces, 
    \[
        \eta:=\max\{\|\mathbf{A}\|_{{L^\infty}(\{|u|<1\})},\|\textup{Ric}_{\bar{g}}\|_{L^\infty(\{|u|<1\})}^{1/2}\}<1/2,
    \]
    where $\mathbf{A}:\{|u|<1\}\rightarrow \mathbb{R}^{n\times n}$ is the second fundamental form of level surfaces of $u$, at a given point (that is to say, if $x\in \{u=\tau\}$, then $\mathbf{A}(x)$ is the second fundamental form of $\tau$-level surface of $u$, at $x$), and $\textup{Ric}_{\bar{g}}$ is Ricci curvature tensor of $(M,\bar{g})$. Moreover, denote $H$ the mean curvature of the level surfaces of $u$ in a similar manner.
    
    Then, for any $\alpha\in (0,1)$, there exists $C=C(n,\alpha,\bar{g})>0$, such that
    \begin{equation}\label{eq: theorem statement}
        \|\mathbf{A}\|_{C^\alpha(B_{1/2}\cap \{u=\tau\})}\leq C \eta,\quad\forall \tau\in(-1,1),
    \end{equation}
    and
    \begin{equation}\label{eq: improved mean curvature bound}
        \|H\|_{C^\alpha(B_{1/2}\cap\{|u|<1\})}\leq C\epsilon^{1-\alpha}\max\{\epsilon,\eta\}\eta.
    \end{equation}
    Furthermore, the second fundamental form and the mean curvature of the free boundary $\partial\{|u|<1\}$ enjoys the same estimates.
\end{theorem}

\begin{remark}
            The theorem provides an improved interior estimate for the mean curvature $H$ as a byproduct. In particular,
            \[
            \|H\|_{L^\infty(B_{1/2}\cap\{|u|<1\})}\lesssim_{n,\alpha,\bar{g}} \epsilon^{1-\alpha}\max\{\epsilon,\eta\}\eta\leq \eta,
            \]
            even though we did not assume any better initial estimates than $\eta$ for $H$. This shows that interfaces converges to the minimal surface in the limit of $\epsilon\rightarrow 0$, in a very strong sense. Moreover, note that $C^\alpha$ norm is in $B_{1/2}\cap\{|u|<1\}$. This means that there holds $C^\alpha$ estimate for $H$ in normal directions to the level surfaces as well, not only within each level surface.
\end{remark}

\begin{remark}
    These uniform estimates are derived from a more fundamental estimate:
    
    \[
        \|\nabla\phi\|_{C^{1,\alpha}(B_{1/2}\cap\{|u|<1\})}\leq C\epsilon^{1-\alpha}\max\{\epsilon,\eta\}\eta,
    \]
    
    where $\phi:=\log(1/|\nabla u|)$. From this estimate, using the identity $\partial_\nu\phi=H$, we obtain the desired results. This estimate is found to be useful in other problems, such as the parabolic Allen-Cahn equation.
\end{remark}

Once this $C^{1,1}$-to-$C^{2,\alpha}$ estimate is established, one can readily follow the compactness argument (see, e.g. \cite[Section 7]{wang2019finite}) that gives a uniform $C^{1,1}$ estimate, provided one has a uniform $C^{1,\alpha}$ bound. That is to say, with the uniform $C^{1,\alpha}$ result of Caffarelli-C\'ordoba \cite{caffarelli2006phase}, we conclude that the transition layers of the free boundary Allen-Cahn equation are uniformly $C^{2,\alpha}$ regular, if they are uniformly Lipschitz graphs.

The proof of the uniform $C^{2,\alpha}$ regularity of the free boundary Allen-Cahn equation is much more elementary than that of classical Allen-Cahn by Wang and Wei \cite{wang2019finite}, which involves many technical details, due to the presence of interactions between interfaces. 

Naturally, to prove such a regularity theorem with free boundaries, the first tool one may come up with would be the hodograph transform. However, the nature of the problem gives thin domain $\{|u|<1\}$ with its thickness comparable to $\epsilon$, and classical hodograph transform arguments that can be applied to Alt-Caffarelli free boundary problem cannot be applied to the free boundary Allen-Cahn equation, to obtain uniform estimates in $\epsilon$.

We present a novel framework to view interfaces as geometric flow, where the notion of time corresponds to the level of $u$. In this framework, we show that the surface normal velocity $\sigma:=1/|\nabla u|$ evolves by the mean curvature $H$ and $\Delta u$ (see Lemma \ref{lem: HCMF general})
\begin{equation}
	\frac{d}{d\tau}\sigma=\sigma^2(H-\sigma\Delta u),
\end{equation}
where the notion of time $\tau$ is the level of the ambient function $u$. Using this ODE and the apriori bound on the mean curvature $H$, one can approximate $\phi:=\log(1/|\nabla u|)=\log\sigma$ from its one dimensional profile,
\begin{equation}\label{eq: estimate from mean curvature bound}
    |\phi-\log\epsilon|\lesssim \epsilon\eta.
\end{equation}
This framework is found to be useful in parabolic settings as well, because it gives a direct pointwise relation between the parabolic Allen-Cahn equations and the mean curvature flow
\begin{equation}
    v=-H+\partial_\nu\phi,
\end{equation}
with an $\epsilon$-decaying bound on $\partial_\nu\phi$ (here $v$ is the surface normal velocity). However, we leave this for our future project and will not discuss in this paper.

Furthermore, an elliptic equation (see Lemma \ref{lem: key elliptic equation})
\begin{equation}\label{eq: elliptic equation by surface energy}
    \Delta\phi=H^2-|\mathbf{A}|^2
\end{equation}
is developed. This equation is in some sense an “intermediate equation" between the equation satisfied by $u$ and Simon's type equation satisfied by $\mathbf{A}$ (see Remark \ref{rmk: parabolic and simon's}), and provides a bridge between the $C^{1,1}$ regularity theory and higher regularity. From this elliptic equation and constructing explicit barrier on $\phi-\log\epsilon$, in the interior (away from the boundary of the cylinder), up to the free boundary, if a priori $C^{1,1}$ bounds $\eta>0$ of the level surfaces are given, we can improve \eqref{eq: estimate from mean curvature bound} to
\begin{equation}\label{eq: estimate from C11 bound}
    |\phi-\log\epsilon|\lesssim \epsilon^2\eta^2.
\end{equation}
Combining the elliptic equation \eqref{eq: elliptic equation by surface energy} and the improved estimate \eqref{eq: estimate from C11 bound}, we directly apply the elliptic boundary regularity estimate by Lian and Zhang \cite{lian2023boundary}, to obtain the uniform $C^{2,\alpha}$ regularity of the free boundaries. Together with a uniform $C^{2,\alpha}$ estimate of the free boundary and a standard global estimate up to the free boundary, we obtain a uniform $C^{2,\alpha}$ bound for all interfaces.

It is possible to study the classical Allen-Cahn equation or other nonlinear elliptic equations 
\[
    \Delta u=f(u)
\]
using the presented framework. However, the $\epsilon$-degenerate nonlinearity $f$ persists in the key equation (see \eqref{eq: key elliptic equation general} and c.f.\eqref{eq: elliptic equation by surface energy}), and this prevents us from directly applying the standard elliptic estimates to obtain $C^{2,\alpha}$ estimates uniform in $\epsilon$. If we can find a suitable approximation of $\sigma=1/|\nabla u|$, we will obtain a straightforward proof of a uniform $C^{2,\alpha}$ estimate for classical Allen-Cahn and other nonlinearity models. We leave this as our future project.

\subsection*{Notations}

We write here a list of symbols used throughout the paper. Note that we use the notation $\sigma$ interchangeably for $\sigma:\Gamma\times[-1,1]\rightarrow\mathbb{R}$ defined by $\frac{1}{|\nabla u\circ F|}$ and $\sigma:\{|u|<1\}\rightarrow\mathbb{R}$ defined by $\frac{1}{|\nabla u|}$, whenever there is no confusion. Same for $\phi=\log\sigma$.

\begin{center}
\begin{tabular}{p{1.5cm}p{11cm}}
$(M,\bar{g})$ & Ambient smooth Riemannian manifold\\
$\text{Ric}_{\bar{g}}$ & Ricci curvature tensor of the ambient manifold $(M,\bar{g})$\\
$u$ & Function as in Theorem \ref{thm: main theorem}\\
$\Gamma$ & $0$-level set of $u$\\
$B_r$ & $n$-dimensional geodesic ball with radius $r>0$ in $(M,\bar{g})$\\
$H$ & Mean curvature of level surfaces of $u$\\
$\mathbf{A}$ & Second fundamental form of level surfaces of $u$\\
$\mathbf{C}$ & $\nabla_{\bar{g}}^2u/|\nabla_{\bar{g}}u|$\\
$F$ & Immersion map of level surfaces of $u$\\
$\nu$ & Unit normal vector of level surfaces of $u$\\
$\sigma$ & $1/|\nabla_{\bar{g}} u\circ F|$ or $1/|\nabla_{\bar{g}} u|$\\
$\phi$ & $\log\sigma=\log(1/|\nabla_{\bar{g}} u\circ F|)$ or $\log(1/|\nabla_{\bar{g}} u|)$\\
$\partial_\nu v$ & $\nabla_{\bar{g}} v\cdot\nu$\\
$\partial_{\nu\nu}v$& $\text{Hess}_{\bar{g}}v(\nu,\nu)$
\end{tabular}
\end{center}

\section{Level Sets as a Geometric Flow}\label{sec: geometric flow}

In this section, we provide an overview of how the level sets of functions can be interpreted as a hyperbolic type geometric flow, using elliptic operators. This relation can have further implications, particularly in free boundary problems and overdetermined problems. Following the proof in \cite{lefloch2008hyperbolic}, the local-in-time well-posedness of this geometric flow can also be established for sufficiently regular initial conditions.

Let $(M,\bar{g})$ be $n+1$ dimensional Riemannian manifold with smooth metric $\bar{g}$. Let $v$ be a smooth function on open connected domain $\Omega\subset M$ without critical points, i.e., $\nabla v\neq 0$. Denote $\Gamma=\{v=0\}$ as the zero level surfaces of  $v$. 

Let $F:\Gamma\times\mathbb{R}\rightarrow\Omega\subset M$ be an immersion of $\tau$-level surface of $v$ in $\mathbb{R}^{n+1}$, defined in the following manner: for each $x\in \Gamma$, $F(x,\cdot):[-1,1]\rightarrow \Omega\subset M$ solves
    \begin{equation}\label{eq: definition of F}
        F(x,0)=x,\textrm{ and }\frac{dF}{d\tau}(x,\tau)=\frac{\nabla_{\bar{g}}v\circ F(x,\tau)}{|\nabla_{\bar{g}} v\circ F(x,\tau)|^2},\quad\forall \tau\in[-1,1],
    \end{equation}
where $\nabla_{\bar{g}} v$ is the gradient of $v$ under ambient metric ${\bar{g}}$. Then
    \[
    \frac{d}{d\tau}v\circ F(x,\tau)=\frac{\nabla_{\bar{g}} v\circ F(x,\tau)\cdot\nabla_{\bar{g}} v\circ F(x,\tau)}{|\nabla_{\bar{g}} v\circ F(x,\tau)|^2}=1,
    \]
and thus, $F$ preserves level surfaces along $\tau$:
\[
v\circ F(x,\tau)=\tau.
\]

Now, consider level surfaces $F(\Gamma,\tau)=\{v=\tau\}$ as a geometric flow in $(M,\bar{g})$, with the immersion metric $g(\tau)$. We denote $\Delta_\Gamma$ as the Laplace-Beltrami operator on the immersed manifold $(\Gamma,g)$ (note that the metric $g$ depends on $\tau$, and thus all related quantities as well). Additionally, we say $\sigma:=1/|\nabla_{\bar{g}} v\circ F|$ is the surface normal velocity, and $\nu:=\frac{\nabla_{\bar{g}} v}{|\nabla_{\bar{g}} v|}$ is the unit normal vector of level surfaces, respectively (increasing direction of $v$ is considered ``outward"). We will use the abbreviation $\mathbf{A}(x,\tau)=\mathbf{A}\circ F(x,\tau)$, $H(x,\tau)=H\circ F(x,\tau)$, and $\nu(x,\tau)=\nu\circ F(x,\tau)$. Note that we choose the sign of the mean curvature such that the mean curvature of sphere is positive (with the unit normal pointing outward). Also recall that the gradient of $v$ is always normal to its level surface, so $F$ is a normal flow, i.e., 
\[
\frac{d}{d\tau}F=\sigma\nu.
\]
Finally, $\mathbf{C}:=\frac{\text{Hess}_{\bar{g}}v}{|\nabla_{\bar{g}} v|}$ is ``extended" second fundamental form, in a sense that when restricted to its tangential components, it is precisely $\mathbf{A}$. That is to say, for any $x\in\Gamma$ and $\tau\in[-1,1]$,
\[
    \mathbf{C}(x,\tau)(\xi,\zeta)=\mathbf{A}(x,\tau)(\xi,\zeta),\quad\forall \xi,\zeta\in T_x\Gamma.
\]

Decomposition of the ambient Laplacian into Laplace-Beltrami operator of the submanifold and its normal derivatives is standard.

\begin{lemma}\label{lem: general Laplace-mean curvature relation}
    For any $f\in C^\infty(\Omega)$, 
    \begin{equation}\label{eq: decomposition of Laplace beltrami of level surface}
        (\Delta_{\bar{g}}f)\circ F=\Delta_\Gamma(f\circ F)+(\partial_{\nu\nu}f)\circ F+H(\partial_\nu f)\circ F,
    \end{equation}
    \begin{equation}\label{eq: timewise derivative relation}
        \frac{1}{\sigma}\frac{d}{d\tau}f\circ F=(\partial_{\nu} f)\circ F,
    \end{equation}
    and
    \begin{equation}
        \left(\frac{1}{\sigma}\frac{d}{d\tau}\right)^2 f\circ F=(\partial_{\nu\nu} f+\mathbf{C}(\nu,\nabla_{\bar{g}} f)-\mathbf{C}(\nu,\nu)\partial_\nu f)\circ F.
    \end{equation}
\end{lemma}
 \begin{proof}
     By applying the Laplace-Beltrami operator to $ f\circ F$, with chain-rule, we have
     \[
         \Delta_\Gamma( f\circ F)=\textrm{tr}_{\bar{g}}((\nabla_{\bar{g}} F)^T((\nabla_{\bar{g}}^2  f)\circ F)\nabla_{\bar{g}} F)+((\nabla_{\bar{g}} f)\circ F)\cdot \Delta_\Gamma F.
     \]
     Since $F(\cdot,\tau)$ is an immersion map of $\Gamma$, $\Delta_\Gamma F=-H\nu$. Moreover, it holds that
     \[
         \textrm{tr}_{\bar{g}}((\nabla_{\bar{g}} F)^T((\nabla_{\bar{g}}^2  f)\circ F)\nabla_{\bar{g}} F)+(\partial_{\nu\nu} f)\circ F=(\Delta_{\bar{g}} f)\circ F.
     \]
     This shows \eqref{eq: decomposition of Laplace beltrami of level surface}. \eqref{eq: timewise derivative relation} follows from \eqref{eq: definition of F}:
     \begin{align*}
         \frac{1}{\sigma}\frac{d}{d\tau} f\circ F&=\frac{\nabla_{\bar{g}}  f\circ F}{\sigma}\cdot \frac{dF}{d\tau}\\
         &=\frac{\nabla_{\bar{g}} f\circ F}{\sigma}\cdot \frac{\nabla_{\bar{g}}v\circ F}{|\nabla_{\bar{g}}v\circ F|^2}\\
         &=\nabla_{\bar{g}} f\circ F\cdot \nu\circ F\\
         &=\partial_\nu f\circ F.
     \end{align*}
     Finally, 
     \begin{align*}
         \left(\frac{1}{\sigma}\frac{d}{d\tau}\right)^2 f\circ F&=(\partial_\nu(\partial_\nu  f))\circ F\\
         &=(\partial_{\nu\nu} f+\partial_\nu\nu\cdot\nabla_{\bar{g}}  f)\circ F\\
         &=(\partial_{\nu\nu} f+\mathbf{C}(\nu,\nabla_{\bar{g}} f)-\mathbf{C}(\nu,\nu)\partial_\nu f)\circ F,
     \end{align*}
     where we have used
     \[
        \partial_\nu\nu=\mathbf{C}(\nu,\cdot)-\mathbf{C}(\nu,\nu)\nu.
     \]
     This completes the proof.
 \end{proof}

Then the level surface of $v$ follows the next evolution equations:

\begin{lemma}[Evolution equations]\label{lem: HCMF general}
     The level surface $F(\Gamma,\tau)=\{v=\tau\}$ satisfies a hyperbolic mean curvature type flow: 
    \begin{equation}\label{eq:HMCF general}
        \frac{d }{d\tau}\sigma=\sigma ^2(H-\sigma \Delta_{\bar{g}} v\circ F).
    \end{equation}
    Moreover,
    \begin{equation}\label{eq: evolution of mean curvature}
        \frac{d}{d\tau}H=-\Delta_\Gamma \sigma-\sigma(|\mathbf{A}|^2+\textup{Ric}_{\bar{g}}(\nu,\nu)),
    \end{equation} 
    where $\textup{Ric}_{\bar{g}}$ is the Ricci curvature tensor of the ambient manifold $(M,{\bar{g}})$.
\end{lemma}

\begin{proof}
    From \eqref{eq: decomposition of Laplace beltrami of level surface}, we have
    \begin{equation}
        0=\Delta_{\Gamma}v=\Delta_{\bar{g}} v+\partial_{\nu\nu}v+H\partial_{\nu}v.
    \end{equation}
    The reason $\Delta_{\Gamma}v=0$ is that $v$ is constant on $\Gamma$. Then we obtain the equation
    \begin{align*}
        \frac{d\sigma }{d\tau}(x,\tau)&=-\frac{((\text{Hess}_{\bar{g}}v)\circ F(x,\tau))(\nu,\nu)}{|\nabla_{\bar{g}} v\circ F(x,\tau)|^3}\\
        &=-\frac{\partial_{\nu\nu}v\circ F(x,\tau)}{(\partial_{\nu}v\circ F(x,\tau))^3}\\
        &=\frac{H(x,\tau)\partial_\nu v\circ F(x,\tau)-\Delta_{\bar{g}} v\circ F(x,\tau)}{(\partial_\nu v\circ F(x,\tau))^3}\\
        &=\sigma^2(x,\tau)\left(H(x,\tau)-\sigma (x,\tau)\Delta_{\bar{g}} v\circ F(x,\tau)\right).
    \end{align*}
    This concludes \eqref{eq:HMCF general}.

    \eqref{eq: evolution of mean curvature} is a structure equation for normal geometric flow, see \cite[Theorem 3.2]{bethuel1999geometric}. This completes the proof.
\end{proof}

\section{Key elliptic equation for \texorpdfstring{$\phi$}{}}

We derive the key elliptic equation satisfied by $\phi=\log\sigma$.

\begin{lemma}[Equation for $\phi$]\label{lem: key elliptic equation}
Suppose $v$ satisfies $\Delta_{\bar{g}}v= f(v)$ with differentiable nonlinearity $ f:\mathbb{R}\rightarrow\mathbb{R}$, and denote $\phi:=\log\sigma=\log(1/|\nabla v|)$. Then
    \begin{equation}\label{eq: key elliptic equation general}
        \Delta_{\bar{g}}\phi=(H-\sigma f)^2-|\mathbf{A}|^2-\text{Ric}_{\bar{g}}(\nu,\nu)- f'(v).
    \end{equation}
    In particular, if $\Delta_{\bar{g}}v=0$, then
    \begin{equation}\label{eq: key elliptic equation harmonic}
        \Delta_{\bar{g}}\phi=H^2-|\mathbf{A}|^2-\text{Ric}_{\bar{g}}(\nu,\nu).
    \end{equation}
\end{lemma}

\begin{remark}\label{rmk: parabolic and simon's}
    More generally, one can derive the parabolic version of this equation. That is to say, if $v$ solves $\partial_tv=\Delta_{\bar{g}} v- f(v)$, then $\phi=\log(1/|\nabla_{\bar{g}} v|)$ solves
    \begin{equation}
        \partial_t\phi=\Delta_{\bar{g}}\phi-\partial_\nu\phi(2\partial_\nu\phi-H+\sigma f)+|\mathbf{A}|^2+\text{Ric}_{\bar{g}}(\nu,\nu)+ f'(u).
    \end{equation}
    Moreover, Simon's type equation can be derived, with the similar argument. In $(M,\bar{g})=\mathbb{R}^{n+1}$, it has the form
    \begin{equation}\label{eq: Simon's type equation}
        \partial_t\mathbf{C}=\Delta \mathbf{C}+\mathbf{C}(|\mathbf{A}|^2+|\nabla_T\phi|^2)-2\nabla\phi\cdot \nabla \mathbf{C}-\sigma f''(u)\nu\otimes\nu,
    \end{equation}
    where $\nabla_T\phi$ is the tangential gradient of $\phi$. These can be used to study parabolic Allen-Cahn equation or higher regularity in phase transitions. However, we will not discuss this further in this paper.
\end{remark}

\begin{proof}
    First, note that 
    \[
        \nabla_{\bar{g}}\sigma=-\frac{\text{Hess}_{\bar{g}}v(\nu,\cdot)}{|\nabla_{\bar{g}}v|^2}=-\sigma \mathbf{C}(\nu,\cdot).
    \]
    Therefore, we obtain
    \begin{equation}\label{eq: key elliptic equation proof 1}
        \mathbf{C}(\nu,\nabla_{\bar{g}}\sigma)=-\frac{|\nabla_{\bar{g}}\sigma|^2}{\sigma}.
    \end{equation}
    Moreover, 
    \begin{equation}\label{eq: key elliptic equation proof 2}
        \mathbf{C}(\nu,\nu)=\frac{\text{Hess}_{\bar{g}}v(\nu,\nu)}{|\nabla_{\bar{g}}v|}=\frac{\Delta_{\bar{g}} v-\text{tr}_{\bar{g}}\mathbf{A}}{|\nabla_{\bar{g}} v|}=\sigma f(v)-H.
    \end{equation}
    Then, using Lemma \ref{lem: general Laplace-mean curvature relation},
    \begin{align*}
        &\Delta_{\bar{g}}\phi\\
        &=\frac{\Delta_{\bar{g}}\sigma}{\sigma}-\frac{|\nabla_{\bar{g}}\sigma|^2}{\sigma^2}\\
        &=\frac{\Delta_\Gamma\sigma+\partial_{\nu\nu}\sigma+H\partial_{\nu}\sigma}{\sigma}-\frac{|\nabla_{\bar{g}}\sigma|^2}{\sigma^2}\\
        &=\frac{\Delta_\Gamma\sigma}{\sigma}+\frac{1}{\sigma}\left(\frac{1}{\sigma}\frac{d}{d\tau}\right)^2\sigma-\frac{\mathbf{C}(\nu,\nabla_{\bar{g}}\sigma)}{\sigma}+\frac{\mathbf{C}(\nu,\nu)\partial_\nu\sigma}{\sigma}+H\frac{\partial_\nu\sigma}{\sigma}-\frac{|\nabla_{\bar{g}}\sigma|^2}{\sigma^2}\\
        &=\frac{\Delta_\Gamma\sigma}{\sigma}+\frac{1}{\sigma^2}\frac{d}{d\tau}\left(\sigma(H-\sigma f)\right)+\mathbf{C}(\nu,\nu)(H-\sigma f(v))+H(H-\sigma f(v))\\
        &=\frac{\Delta_\Gamma\sigma}{\sigma}+(H-\sigma f(v))^2+\frac{-\Delta_\Gamma \sigma-\sigma|\mathbf{A}|^2-\sigma\text{Ric}_{\bar{g}}(\nu,\nu)}{\sigma}-(H-\sigma f(v)) f(v)- f'(v)\\
        &\hspace{4mm}-(H-\sigma f(v))^2+H(H-\sigma f(v))\\
        &=(H-\sigma f(v))^2-|\mathbf{A}|^2-\text{Ric}_{\bar{g}}(\nu,\nu)- f'(v),
    \end{align*}
    where we have used \eqref{eq:HMCF general}, \eqref{eq: evolution of mean curvature}, \eqref{eq: key elliptic equation proof 1}, \eqref{eq: key elliptic equation proof 2}. This completes the proof.
\end{proof}

\section{Proof of Theorem \ref{thm: main theorem}}\label{sec: improved gradient estimate}
In this section, we denote $u$ as in Theorem \ref{thm: main theorem}. Moreover, we denote $F$, $\sigma=1/|\nabla u|$, and $\phi=\log(1/|\nabla u|)$ as in Section \ref{sec: geometric flow}. First, we estimate $\sigma-\epsilon$, which directly follows from the estimate on the mean curvature.
\begin{lemma}
    Let $u$ be as in Theorem \ref{thm: main theorem}. Then we have
    \begin{equation}\label{eq: naive gradient estimate in sigma}
        |\sigma-\epsilon|\leq \epsilon^2\eta,\quad\textup{in}\quad\{|u|<1\},
    \end{equation}
    or in other words,
    \begin{equation}\label{eq: naive gradient estimate}
       |\phi-\log\epsilon|\leq \epsilon\eta,\quad\textup{in}\quad\{|u|<1\}.
    \end{equation}
\end{lemma}

\begin{proof}
    Since $u$ is harmonic, from \eqref{eq:HMCF general}, we have
    \begin{equation}
        \frac{d}{d\tau}\sigma=\sigma^2H.
    \end{equation}
    Therefore the conclusion follows immediately from solving this ODE with the initial condition (the free boundary gradient condition) $\sigma(\cdot,\pm 1)=\epsilon$ and $C^{1,1}$ bound $|H|\leq \eta$.
\end{proof}

\begin{remark}
	Note that to get this a priori estimate, we do not require $C^{1,1}$ bounds on the level surfaces. The only estimate we need is the bound on the mean curvature $\epsilon\|H\|_{L_\tau^\infty(-1,1)}<1$.
\end{remark}

Because we also have $C^{1,1}$ bound (not only the mean curvature bound) of level surfaces, we can improve the interior gradient estimate by a barrier argument. This improvement was made possible by the observation that $\phi=\log(1/|\nabla u|)$ solves a Poisson equation 
\[
    \Delta\phi=H^2-|\mathbf{A}|^2
\]    
with the square norm of the traceless second fundamental form on the right hand side. Additionally, from the free boundary condition, $\phi$ satisfies the Dirichlet boundary condition $\phi-\log\epsilon=0$ on the free boundary. Therefore, we have reduced the overdetermined problem on $u$ to Dirichlet boundary value problem for $\phi$. 

\begin{lemma}\label{lem: interior gradient estimate}
    Let $u$ be as in Theorem \ref{thm: main theorem}. If we denote $\phi=\log(1/|\nabla u|)$, then we have
    \begin{equation}\label{eq: interior gradient estimate}
        |\phi-\log\epsilon|\leq C\epsilon^2\max\{\epsilon,\eta\}\eta,\quad\textrm{in}\quad B_{3/4}\cap\{|u|<1\},
    \end{equation}
    for some $C=C(n,\bar{g})>0$.
\end{lemma}

\begin{proof}
    Recall that $x_0$ (the center of the geodesic ball $B_1$) is on $\Gamma=\{u=0\}$. Combine \eqref{eq: key elliptic equation harmonic} with the free boundary condition and \eqref{eq: naive gradient estimate}, $\phi-\log\epsilon$ solves the following elliptic problem:
    \begin{equation}\label{eq: elliptic equation satisfied by phi}
    \left\{
    \begin{alignedat}{3}
        \Delta_{\bar{g}} (\phi-\log\epsilon)&=H^2-|\mathbf{A}|^2-\text{Ric}_{\bar{g}}(\nu,\nu)&\quad
        &\textrm{in}&&\quad\{|u|<1\},\\
        \phi-\log\epsilon&=0&\quad&\textrm{on}&&\quad\partial\{|u|<1\},\\
        \left|\phi-\log\epsilon\right|&\leq  \epsilon\eta&\quad&\textrm{on}&&\quad\partial B_1\cap\overline{\{|u|<1\}}.
    \end{alignedat}
    \right.
    \end{equation}
    
    We define the supersolution by
    \[
        \Phi(x):=C_\tau\epsilon^2(1-u^2(x))+C_xd_{\bar{g}}^2(x_0,x),
    \]
    where $d_{\bar{g}}$ is the geodesic distance, and $C_\tau,C_x>0$ are constants to be chosen later. From the harmonicity of $u$, we can compute the Laplacian of $\Phi$: 
    \[
        \Delta_{\bar{g}} u^2=2|\nabla_{\bar{g}} u|^2,
    \]
    thus 
    \[
        \Delta_{\bar{g}}\Phi\leq 2nC_x(1+C_{\bar{g}})-2C_\tau\epsilon^2|\nabla_{\bar{g}} u|^2,
    \]
    with $C_{\bar{g}}$ is a constant only depending on the ambient metric $\bar{g}$. Now, let $C_\tau=2n(4\max\{\epsilon,\eta\}\eta(1+C_{\bar{g}})+\eta^2)$ and $C_x=4\max\{\epsilon,\eta\}\eta$. With these choices, we find that
    \[
        \Delta_{\bar{g}} \Phi<  - 2n\eta^2\leq \Delta_{\bar{g}} (\phi-\log\epsilon),
    \]
    because of $C^{1,1}$ and the ambient curvature bounds $|H^2-|\mathbf{A}|^2-\text{Ric}_{\bar{g}}(\nu,\nu)|\leq 2n\eta^2$. Moreover, it satisfies the boundary condition on the free boundary
    \[
        \Phi=C_xd_{\bar{g}}^2(x_0,\cdot)\geq 0,\quad\text{on}\quad\partial\{|u|<1\},
    \]
    and also on the boundary of the ball
    \[
        \Phi\geq C_x\geq 4\epsilon\eta\geq \phi-\log\epsilon,\quad\textrm{on}\quad \partial B_1\cap\overline{\{|u|<1\}}.
    \]

    Therefore, we deduce that $\Phi$ is a supersolution to $\phi-\log\epsilon$. The construction of a subsolution follows similarly, and by applying the maximum principle, we conclude that
    \[
    \left|\phi(x_0)-\log\epsilon\right|= C_\tau\epsilon^2(1-u^2(x_0))\lesssim_{n,\bar{g}} \epsilon^2\max\{\epsilon,\eta\}\eta.
    \]

    By repeating the similar argument for all other points in $B_{3/4}\cap\{|u|<1\}$, we arrive at the conclusion.
\end{proof}

Using the improved estimate and elliptic equation \eqref{eq: elliptic equation satisfied by phi}, we can first provide the uniform $C^{2,\alpha}$ estimate of the free boundary.

\begin{lemma}\label{eq: C2alpha estimate of the free boundary}
    Let $H$ and $\mathbf{A}$ as in Theorem \ref{thm: main theorem}. There exists $C=C(n,\alpha,\bar{g})>0$ such that
    \[
        \|H\|_{C^\alpha(B_{1/2}\cap\partial\{|u|<1\})}\leq C\epsilon^{1-\alpha}\eta^2,
    \]
    and
    \begin{equation}\label{eq: free boundary c2alpha estimate}
        \|\mathbf{A}\|_{C^\alpha(B_{1/2}\cap\partial\{|u|<1\})}\leq C\eta.
    \end{equation}
\end{lemma}

\begin{proof}
    Recall that $\phi=\log\sigma=\log(1/|\nabla u|)$ satisfies \eqref{eq: elliptic equation satisfied by phi}. Also recall \eqref{eq: interior gradient estimate}. Applying pointwise $C^{1,\alpha}$ boundary estimate (e.g. see \cite[Theorem 1.15]{lian2023boundary}) to $\phi-\log\epsilon$, we obtain
    \[
       \|\nabla_{\bar{g}} \phi\|_{C^{\alpha}(B_{3/4}\cap \partial\{|u|<1\})}\leq C(n,\alpha,\bar{g}) \epsilon^{1-\alpha}\max\{\epsilon,\eta\}\eta.
    \]
    Then from \eqref{eq:HMCF general} and \eqref{eq: timewise derivative relation} we derive
    \[
        H=\partial_\nu \phi.
    \]
    Therefore, we have $C^\alpha$ estimate of the mean curvature of the free boundary:
    \[
        \|H\|_{C^\alpha(B_{3/4}\cap\partial\{|u|<1\})}\leq C(n,\alpha,\bar{g})\epsilon^{1-\alpha}\max\{\epsilon,\eta\}\eta.
    \]
    By the standard interior estimate from $H$, we have
    \[
        \|\mathbf{A}\|_{C^\alpha(B_{1/2}\cap\partial\{|u|<1\})}\leq C(n,\alpha,\bar{g})\eta.
    \]
    This concludes the proof.
\end{proof}

Finally, we complete the proof by using standard elliptic estimate up to the free boundary.

\begin{proof}[Proof of Theorem \ref{thm: main theorem}]
    Now, with the uniform $C^{2,\alpha}$ bound on the free boundaries (Lemma \ref{eq: C2alpha estimate of the free boundary}), from the standard elliptic estimates (see, e.g., \cite[(4.46)]{gilbarg2015elliptic}), we obtain the uniform $C^{1,\alpha}$ estimate on $\phi$
    \[
        \|\nabla_{\bar{g}} \phi\|_{C^\alpha(B_{3/4}\cap \{|u|=\tau\})}\leq C(n,\alpha,\bar{g})\epsilon^{1-\alpha}\max\{\epsilon,\eta\}\eta,\quad\forall \tau\in (-1,1),
    \]
    from the elliptic equation
    \[
        \left\{
        \begin{alignedat}{2}
            \Delta (\phi-\log\epsilon)&=H^2- |\mathbf{A}|^2-\text{Ric}_{\bar{g}}(\nu,\nu)\quad
        &\textrm{in}&\quad\{|u|<1\},\\
        \phi-\log\epsilon&=0\quad&\textrm{on}&\quad\partial\{|u|=\tau\}.
        \end{alignedat}
        \right.
    \]
    Then the claim of the theorem follows as in the proof of Lemma \ref{eq: C2alpha estimate of the free boundary}, using $H=\partial_\nu \phi$.
\end{proof}

\section*{Acknowledgements}

The author would like to thank Prof. Joaquim Serra for his invaluable guidance and mentorship during the course of this research, which originated as part of the author’s Master’s thesis project. The author is also deeply grateful to Dr. Xavier Fernandez-Real for his encouragement and insightful advice, which motivated the completion of this work. Profound gratitude is reserved for Dr. Hardy Chan, whose invaluable advice and proofreading significantly enriched the quality of this work. This research was supported by the Swiss National Science Foundation (SNF grant number: PZ00P2\_202012).

\bibliographystyle{abbrvnat}
\bibliography{Bib.bib}

\medskip

\end{document}